\documentclass[a4paper, 12pt]{amsart}
\usepackage[left=3cm,right=3cm]{geometry}

\usepackage{qmod}

\usepackage[english]{babel}

\usepackage{amsthm, amssymb, amsfonts, amsmath, mathrsfs}
\usepackage[all]{xy}

\usepackage{caption}
\usepackage{url}
\usepackage{tikz}
\usepackage{hyperref}

\usepackage{xcolor}

\usepackage{array}
\newcommand{\PreserveBackslash}[1]{\let\temp=\\#1\let\\=\temp}
\newcolumntype{C}[1]{>{\PreserveBackslash\centering}p{#1}}
\newcolumntype{R}[1]{>{\PreserveBackslash\raggedleft}p{#1}}
\newcolumntype{L}[1]{>{\PreserveBackslash\raggedright}p{#1}}

\newcounter{stepcounter}
\theoremstyle{plain}

\newtheorem{thm}{Theorem}[section]

\newtheorem{lem}[thm]{Lemma}

\theoremstyle{definition}

\newtheorem{defn}[thm]{Definition}

\newtheorem{remark}[thm]{Remark}

\newcommand\bit{\begin{itemize}}
\newcommand\eit{\end{itemize}}
\newcommand\bet{\begin{enumerate}}
\newcommand\eet{\end{enumerate}}
\newcommand\ed{\end{document}}

\DeclareFontFamily{U}{mathx}{\hyphenchar\font45}
\DeclareFontShape{U}{mathx}{m}{n}{
      <5> <6> <7> <8> <9> <10>
      <10.95> <12> <14.4> <17.28> <20.74> <24.88>
      mathx10
      }{}
\DeclareSymbolFont{mathx}{U}{mathx}{m}{n}
\DeclareFontSubstitution{U}{mathx}{m}{n}
\DeclareMathAccent{\widecheck}{0}{mathx}{"71}
\DeclareMathAccent{\wideparen}{0}{mathx}{"75}

\newcommand{\e}{\varepsilon}
\newcommand\Om{\Omega}
\newcommand\del{\partial}
\newcommand\adel{\ol{\partial}}
\newcommand\DEL{\Delta}

\newcommand\bC{{\mathbb C}}
\newcommand\bN{{\mathbb N}}

\newcommand\F{{\mathcal F}}

\newcommand\N{{\mathcal N}}

\newcommand\co{\mathrm{co}}

\newcommand\exd{\mathrm{d}}

\newcommand\unit{\mathrm{U}}
\newcommand\counit{\mathrm{C}}

\newcommand\id{\mathrm{id}}
\newcommand\proj{\mathrm{proj}}

\newcommand\coby{\, \square_{H}}
\newcommand\oby{\otimes}

\newcommand\sseq{\subseteq}
\newcommand\tl{\triangleleft}

\def\qbinom#1#2{\ensuremath{\left[\kern-.3em\left[\genfrac{}{}{0pt}{}{#1}{#2}\right]\kern-.3em\right]_q}}

\newcommand\ol{\overline}

\newcommand\mto{\mapsto}

\usepackage{tikz}
\usetikzlibrary{decorations.pathreplacing}

\usepackage[english]{babel}

\newcommand{\fl}{\mathfrak{l}}

\newcommand{\fsl}{\mathfrak{sl}}
\newcommand{\fsu}{\mathfrak{su}}
\newcommand{\eps}{\varepsilon}

\newcommand{\cP}{\mathcal{P}}
\newcommand{\cO}{\mathcal{O}}
\newcommand{\cE}{\mathcal{E}}
\newcommand{\cF}{\mathcal{F}}

\usepackage{mathtools}

\DeclareMathOperator{\Span}{span}

\author[A. O. Krutov]{Andrey O. Krutov}
\address{Mathematical Institute of Charles University, Sokolovsk\'a 83, Prague, Czech Republic}
\email{andrey.krutov@matfyz.cuni.cz}

\author[R. \'O Buachalla]{R\'eamonn \'O Buachalla}
\address{Mathematical Institute of Charles University, Sokolovsk\'a 83, Prague, Czech Republic}
\email{obuachalla@karlin.mff.cuni.cz}

\title[Curvature of relative line modules over quantum projective spaces]{Curvature of positive relative line modules over the quantum projective spaces}

\thanks{A.K. was supported by the GA\v{C}R Grant EXPRO 19-28628X.
  R.\'OB was supported by the Charles University PRIMUS grant \emph{Spectral Noncommutative Geometry of Quantum Flag Manifolds} PRIMUS/21/SCI/026.
  This article is based upon work from COST Action CaLISTA CA21109 supported by COST (European
  Cooperation in Science and Technology, \url{www.cost.eu}).
}
\keywords{quantum groups, noncommutative geometry, quantum flag manifolds, complex geometry}

\date{}

\subjclass[2020]{
  46L87, 
  81R60, 
  81R50, 
  17B37, 
  16T05}  

\begin{document}

\maketitle

\begin{abstract}
  We show that the curvature of a positive relative line module over quantum projective space is given by $q$-integer deformation of its classical curvature. This generalises a result of Majid for the Podle\'s sphere.
\end{abstract}


\section{Introduction}
For any  complex manifold, the curvature of the Chern connection is additive over tensor products of
holomorphic vector bundles. In particular, any tensor power of a~positive line bundle is again positive.
In the noncommutative setting tensoring two holomorphic vector bundles is more problematic. First of all, in
order to define a~tensor product, at least one of the bundles noncommutative holomorphic structures  $(\cF,\adel_{\cF})$ needs to be a~bimodule.
Moreover,  $\adel_{\cF}$  needs to be a bimodule connection (in the sense of
\cite{DVMadoreMouradBimodule,MadoreBimodule,DVMichor}). Even in this case, curvature does not behave
additively.
In particular, for a~bimodule holomorphic line bundle~$\cE$,
we cannot directly conclude positivity of~$\cE^{\otimes_B k}$ from positivity of~$\cE$,
making the general approach of~\cite{Fano} all the more valuable.
For the case of~$\cO_q(S^2)$ the Podle\'s sphere~\cite{Maj}, and~$\cO_q(\mathbb{CP}^2)$ the quantum projective
plane~\cite{Antiselfdual}, it is known that the classical line bundle curvatures $q$-deform to quantum integer
curvatures.
In this paper we show that this process generalises to all positive line bundles over the quantum projective spaces.
This suggests that there is some type of $q$-deformed (or braided) additivity underlying these results.
Understanding this process presents itself as an~interesting and important future goal.
Note that throughout this paper~$(k)_q$ denotes  the quantum integer defined in
Appendix~\ref{app:quantumintegers}.

Recall from~\cite{Fano}, that the positive homogeneous line bundles over the irreducible quantum flag manifolds, in
particular, quantum projective space, are indexed by positive integers. In this paper we show that the
curvature of a~positive homogeneous line bundle~$\cE_k$, where $k\in\mathbb{N}$, over a~quantum projective
space~$\cO_q(\mathbb{CP}^n)$ is given by the $q$-integer~$(k)_{q^{-2/(n+1)}}$.

In future work, it is planned to extend this work to include all of the irreducible quantum flag manifolds endowed with their Heckenberger--Kolb calculi.



\section{Quantum Principal Bundles and Connections}

In this section, we recall those definitions and results from the theory of quantum principal bundles
necessary for our explicit curvature calculations below.

For recent advances in the theory of quantum principal bundles, see~\cite{AschieriFioresiLatini2021,Branimir2021CMP,BrzezinskiSzymanski2021,NicholsGrass} and
references therein.

\subsection{Universal Differential Calculi}

A {\em first-order differential calculus} over an~algebra~$B$ is a~pair $(\Om^1,\exd)$, where $\Omega^1$ is
a~$B$-bimodule and $\exd: B \to \Omega^1$ is a~linear map for which the {\em Leibniz rule} holds 
\begin{align*}
  \exd(ab)=a(\exd b)+(\exd a)b,&  & a,b \in B,
\end{align*}
and for which~$\Om^1$ is generated as a~left $B$-module by those elements of the form~$\exd b$, for~$b \in B$.
The {\em universal first-order differential calculus} over~$B$ is the pair 
$(\Om^1_u(B), \exd_u)$, where $\Om^1_u(B)$ is the kernel of the multiplication map $m_B: B \otimes B \to B$ endowed
with the obvious bimodule structure, and $\exd_u$ is the map defined by
\begin{align*}
  \exd_u: B \to \Omega^1_u(B), & & b \mto 1 \oby b - b \oby 1.
\end{align*}
By \cite[Proposition 1.1]{WoroCQPGs}, every first-order differential calculus over $B$ is of the form
$\left(\Omega^1_u(B)/N, \,\proj \circ \exd_u\right)$, where $N$ is a~$B$-sub-bimodule of $\Omega^1_u(B)$, and
we denote by $\proj:\Omega^1_u(B) \to \Omega^1_u(B)/N$ the quotient map. This gives a~bijective correspondence
between calculi and sub-bimodules of $\Omega^1_u(B)$.

For $A$ a~Hopf algebra, and $B$ a~left $A$-comodule algebra, we say that a~first-order differential calculus
$\Omega^1(B)$ over $B$ is   {\em left $A$-covariant} if there exists a (necessarily unique) map
$\DEL_L:\Om^1(B) \to A \oby \Om^1(B)$ satisfying 
\begin{align*}
  \DEL_L(b\exd b') = \DEL_L(b) (\id \oby \exd)\DEL_L(b'),  &  & b,b' \in B.
\end{align*}
For the special case of $A$ considered as a left $A$-comodule algebra over itself, we note that every
covariant first-order differential calculus over $A$ is naturally an object in the category
$\qMod{A}{A}{}{A}$; see \S\ref{app:fun}.

\subsection{Holomorphic structures}

\subsubsection{Connections}
Motivated by the Serre--Swan theorem~\cite{Serre1955,Swan1962}, we think of a finitely generated projective left $B$-module $\F$ as
a~noncommutative generalisation of a~vector bundle.
For $\Omega^\bullet$ a~differential calculus over an~algebra~$B$ and~$\mathcal{F}$ a~finitely generated
projective left $B$-module, a \emph{connection} on~$\F$ is a~$\mathbb{C}$-linear map 
\[
  \nabla:\mathcal{F} \to \Omega^1 \otimes_B \F
\]
satisfying 
\begin{align} \label{eqn:connLR}
  \nabla(bf) = \exd b \otimes f + b \nabla f, & &\text{for all } b \in B, f \in \F.
\end{align}

Any connection can be extended to a map $\nabla: \Omega^\bullet \otimes_B \mathcal{F} \to   \Omega^\bullet \otimes_B \mathcal{F}$ uniquely defined by 
\begin{align*}
  \nabla(\omega \otimes f) =   \exd \omega \otimes f + (-1)^{|\omega|} \, \omega \wedge \nabla f,
\end{align*}
where $f \in \F$, and~$\omega$ is a~homogeneous element of $\Omega^{\bullet}$ of degree~$|\omega|$.
The \emph{curvature} of a~connection is the left $B$-module map
$\nabla^2: \mathcal{F} \to \Omega^2 \otimes_B\mathcal{F}$.
A~connection is said to be {\em flat} if $\nabla^2 = 0$.

\subsubsection{Complex Structures}

In this subsection we recall the definition of a~complex structure for a differential calculus, as introduced
in~\cite{KLvSPodles, BS}, see also~\cite{BeggsMajid:Leabh}. This gives an abstract characterisation of the
properties of the de Rham complex of a~classical complex manifold~\cite{HUY}.

\begin{defn}\label{defnnccs}
  A {\em  complex structure} $\Om^{(\bullet,\bullet)}$, for a~differential $*$-calculus~$(\Om^{\bullet},\exd)$,
  is an $\bN^2_0$-algebra grading $\bigoplus_{(a,b)\in \bN^2_0} \Om^{(a,b)}$
  for $\Om^{\bullet}$ such that, for all $(a,b) \in \bN^2_0$: 
  \[
    \Om^k = \bigoplus_{a+b = k} \Om^{(a,b)},\qquad
    \big(\Om^{(a,b)}\big)^* = \Om^{(b,a)},\qquad
    \exd \Om^{(a,b)} \sseq \Om^{(a+1,b)} \oplus \Om^{(a,b+1)}.
  \]

\end{defn}

An element of~$\Om^{(a,b)}$ is called an~\emph{$(a,b)$-form}.
For $ \proj_{\Om^{(a+1,b)}}$, and~$ \proj_{\Om^{(a,b+1)}}$,
the projections from $\Om^{a+b+1}$ to~$\Om^{(a+1,b)}$, and~$\Om^{(a,b+1)}$ respectively, we write
\begin{align*}
  \del|_{\Om^{(a,b)}} : = \proj_{\Om^{(a+1,b)}} \circ \exd, & & \ol{\del}|_{\Om^{(a,b)}} : = \proj_{\Om^{(a,b+1)}} \circ \exd.
\end{align*}
It follows from Definition~\ref{defnnccs} that for any complex structure, 
\begin{align*}
  \exd = \del + \adel, & &  \adel \circ \del = - \, \del \circ \adel, & & \del^2 = \adel^2 = 0. 
\end{align*}
Thus $\big(\bigoplus_{(a,b)\in \bN^2_0}\Om^{(a,b)}, \del,\ol{\del}\big)$ is a~double complex.
Both~$\del$ and~$\adel$ satisfy the graded Leibniz rule. Moreover,   
\begin{align*} \label{eqn:stardel}
  \del(\omega^*) = \big(\adel \omega\big)^*, & &  \adel(\omega^*) = \big(\del \omega\big)^*, & & \text{ for all  } \omega \in \Om^\bullet. 
\end{align*}
Associated with any complex structure $\Omega^{(\bullet,\bullet)}$ we have a~second complex structure, called
its {\em opposite complex structure}, defined as 
\[
  \ol{\Om}^{(\bullet,\bullet)} := \bigoplus_{(a,b) \in \bN^2_0} \ol{\Om}^{(a,b)},
  \qquad\qquad
  \text{where $\ol{\Om}^{(a,b)}:= \Omega^{^{(b,a)}}$}.  
\]
See~\cite[\textsection 1]{BeggsMajid:Leabh} or~\cite{MMF2} for a~more detailed discussion of complex structures.

\subsubsection{Holomorphic Modules}

In this subsection we present the notion a~holomorphic left $B$-module for an~algebra~$B$.
Such a~module should be thought of as a~noncommutative holomorphic vector bundle,
as has been considered in a number of previous papers, see for example~\cite{BS}, \cite{PolishSch},
\cite{KLvSPodles}, \cite{HolVBs}, and~\cite{Fano}.
Indeed, the definition for holomorphic modules is motivated by the classical Koszul--Malgrange
characterisation of holomorphic bundles~\cite{KoszulMalgrange}.
See~\cite{OSV} for a more detailed discussion.

With respect to a~choice $\Omega^{(\bullet,\bullet)}$ of complex structure on~$\Omega^{\bullet}$,
a~\emph{$(0,1)$-connection on $\mathcal{F}$} is a~connection with respect to the differential calculus~$(\Omega^{(0,\bullet)},\adel)$.

\begin{defn}
  Let $(\Omega^\bullet, \exd)$ be a~differential $*$-calculus over a $*$-algebra $B$,
  equipped with a~complex structure $\Omega^{(\bullet, \bullet)}$.
  A~\emph{holomorphic} left $B$-module is a~pair $(\mathcal{F},\adel_{\mathcal{F}})$,
  where $\mathcal{F}$ is a~finitely generated projective left $B$-module,
  and $\adel_{\mathcal{F}}: \mathcal{F} \to \Omega^{(0,1)} \otimes_B \mathcal{F}$
  is a~flat $(0,1)$-connection.
  We call~$\adel_{\F}$ the \emph{holomorphic structure} of the holomorphic left $B$-module. 
\end{defn}


\subsection{Quantum Principal Bundles and Principal Connections} \label{sec:qpbqhs}

We say that a~right $H$-comodule algebra $(P,\DEL_R)$ is a \emph{Hopf--Galois extension} of $B:=P^{\co(H)}$ if
an isomorphism is given by
\begin{align*}
  \mathrm{can}:= (m_P \otimes \id) \circ( \id \otimes \Delta_R) :P \otimes_B P \to P \otimes H, & & r \otimes s \mapsto rs_{(1)} \otimes s_{(2)},
\end{align*}
where $m_P$ denotes the multiplication in~$P$. It was shown~\cite[Proposition~3.6]{TBGS} that~$P$ is
a~Hopf--Galois extension of $B=P^{\co(H)}$ if and only if an~exact sequence is given by
\begin{align} \label{qpbexactseq}
  0 \longrightarrow P\Om^1_u(B)P {\buildrel \iota \over \longrightarrow} \Om^1_u(P) {\buildrel {\overline{\text{can}}}\over \longrightarrow} P \oby H^+ \longrightarrow 0,
\end{align}
where $\Om^1_u(B)$ is the restriction of $\Om^1_u(P)$ to $B$, we denote by $\iota$ the inclusion, and 
\(
\overline{\mathrm{can}}
\)
is the restriction to  $\Om^1_u(P)$ of the map
\[
(m_P \oby \id) \circ (\id \oby \DEL_R) : P \oby P \to P \oby H.
\]
(Note that the map's domain of definition is $P \otimes P$, rather than $P \otimes_B P$.)
The following definition, due to Brzezi\'nski and Majid~\cite{TBSM1,TBSMRevisited},
presents sufficient criteria for the existence of a~non-universal version of this sequence.
A~non-universal calculus on~$P$ is said to be {\em right $H$-covariant} if the following, necessarily unique, map is well defined
\begin{align*}
  \DEL_R: \Om^1(P) \to \Om^1(P) \otimes H, & & r \exd s \mto r_{(0)} \exd s_{(0)} \oby r_{(1)} s_{(1)}.
\end{align*}

\begin{defn} \label{qpb} Let $H$ be a~Hopf algebra. A {\em quantum principal $H$-bundle} is a~pair
  $(P,\Omega^1(P))$, consisting of a~right $H$-comodule algebra $(P,\Delta_R)$, such that~$P$ is
  a~Hopf--Galois extension of $B = P^{\,\co(H)}$, together with a~choice of right-$H$-covariant calculus
  $\Om^1(P)$, such that for $N \sseq \Om^1_u(P)$ the corresponding sub-bimodule of the universal calculus, we
  have $\overline{\mathrm{can}}(N) = P \oby I$, for some $\mathrm{Ad}$-sub-comodule right ideal
  \[
    I \sseq H^+ :=  \ker(\e: H \to \mathbb{C}),
  \]
  where $\mathrm{Ad} : H \to H \otimes H$ is defined by $\mathrm{Ad}(h) := h_{(2)} \otimes S(h_{(1)}) h_{(3)}$.
\end{defn}

Denoting by~$\Om^1(B)$ the restriction of $\Om^1(P)$ to~$B$, and $\Lambda^1_H := H^+/I$,
the quantum principal bundle definition implies that an exact sequence is given by
\begin{align} \label{Eqn:qpbexactseq}
  0 \longrightarrow P\Om^1(B)P {\buildrel \iota \over \longrightarrow} \, \Om^1(P) {\buildrel {~\overline{\mathrm{can}}~~}\over \longrightarrow} P \oby \Lambda^1_H \longrightarrow 0,
\end{align}
where by abuse of notation, $\overline{\mathrm{can}}$ denotes the map induced on~$\Om^1(P)$ by
identifying~$\Om^1(P)$ as a~quotient of~$\Om_u^1(P)$ (for details see~\cite{PHC}).

A {\em principal connection} for a~quantum principal $H$-bundle $(P,\Omega^1(P))$ is a~right $H$-comodule,
left $P$-module, projection $\Pi:\Om^1(P) \to \Om^1(P)$ satisfying
\[
\ker(\Pi) = P\Om^1(B)P.
\]
The existence of a~principal connection is equivalent to the existence of a~left \linebreak $P$-module, right $H$-comodule,
splitting of the exact sequence given in~\eqref{Eqn:qpbexactseq}.
A~principal connection~$\Pi$ is called {\em strong} if $(\id - \Pi) \big(\exd P\big) \sseq \Om^1(B)P$.

\subsection{Quantum Principal Bundles and Quantum Homogeneous Spaces}

We now restrict to the case of a~homogeneous quantum principal bundle, which is to say, a~quantum principal
bundle whose composite $H$-comodule algebra is a~\emph{quantum  homogeneous space} $B:=A^{\co(H)}$, given by
a~surjective Hopf algebra map $\pi : A \to H$.
For this special case,  it is natural to restrict to calculi on~$A$ which are left $A$-covariant.
Any such calculus~$\Omega^1(A)$ is an~object in $\qMod{A}{A}{}{A}$, and so,
by the fundamental theorem of two-sided Hopf modules (Appendix~\ref{app:fun}), we have the isomorphism
\[
  \unit: \Omega^1(A) \simeq A \otimes  F\!\left(\Omega^1(A)\right)\!.
\]
As a~direct calculation will verify, with respect to the right $H$-coaction 
\begin{align*}
  F(\Omega^1(A)) \to F(\Omega^1(A)) \otimes H, & & [\omega] \mapsto  [\omega_{(0)}] \otimes \pi(S(\omega_{(-1)})) \omega_{(1)},
\end{align*}
the unit $\unit$ of the equivalence is a~right $H$-comodule map.
(Here the right \mbox{$H$-coaction} on $A \otimes F(\Omega^1(A))$ is the usual tensor product coaction.)
Thus a~left $A$-covariant principal connection is equivalent to a~choice of right $H$-comodule decomposition
\[
  F(\Omega^1(A)) \simeq F(A\Omega^1(B)A) \oplus F(A \otimes \Lambda^1_H) \simeq F(A\Omega^1(B)A) \oplus \Lambda^1_H.
\]
As established in~\cite{BWGrass}, for a~homogeneous quantum principal bundle with cosemisimple
composite Hopf algebras, all principal connections are strong.
 
Next, we come connections on any $\cF \in \modz{A}{B}$. Note first that we have a~natural embedding 
\begin{align*}
  j: \Om^1(B) \oby_B \cF \hookrightarrow  \Om^1(B)A \, \square_H \Phi(\cF), & & \omega \otimes f \mapsto \omega f_{(-1)} \otimes [f_{(0)}].
\end{align*}
We claim that a~strong principal connection~$\Pi$ defines a~connection~$\nabla$ on~$\cF$ by
\begin{align*}
  \nabla: \cF \to  \Om^1(B) \oby_B \cF, &   & f \mto  j^{-1} \big((\id - \Pi)(\exd f_{(-1)}) \oby [f_{(0)}]\big).
\end{align*}
Indeed, since~$\exd$ and the projection~$\Pi$ are both right \mbox{$H$-comodule} maps,
their composition $(\id - \Pi) \circ \mathrm{d}$ is a~right $H$-comodule map.
Hence the image of $(\id - \Pi) \circ \mathrm{d}$  is contained in  $j\left(\Om^1(B) \oby_B \cF\right)$,
and~$\nabla$ defines a~connection.
Moreover, if the principal connection~$\Pi$ is a~left $A$-comodule map,
then the connection~$\nabla$ is also a~left $A$-comodule map.

\section{Curvature of Line Bundles over Quantum Projective Spaces}
In this section we compute curvature of positive homogeneous line bundles over quantum projective space using
the framework of quantum principal bundles.
\subsection{Quantum Projective Spaces}
In this subsection we recall the definition of {\em quantum projective space},
which is to say the $A$-series irreducible quantum flag manifold $\cO_q(\mathbb{CP}^n)$.
For details and notation see~\cite{HolVBs,Fano} or~\cite{DOS1}.

Fix~$q\in\mathbb{C}$ such that $q\neq0$ and $q^n\neq1$ for any integer $n\geq1$.
The \emph{quantised universal enveloping algebra~$U_q(\fsl_{n+1})$} is the associative algebra generated by the
elements $E_i$, $F_i$, $K_i$, and $K^{-1}_i$, for $i = 1,\ldots,n$, subject to the relations
\begin{gather*}
  K_iE_jK_i^{-1} = q^{2\delta_{i,j}-\delta_{i,j-1} - \delta_{i,j+1}} E_j,\qquad
  K_iF_jK_i^{-1} = q^{-2\delta_{i,j}+\delta_{i,j-1} - \delta_{i,j+1}} F_j,\qquad
  K_iK_j=1,\\
  K^{-1}_iK_i=K_iK^{-1}_i=1,\qquad
  E_iF_j - F_jE_i = \delta_{i,j}\frac{K_i-K^{-1}_i}{q-q^{-1}},
\end{gather*}
together with the \emph{quantum Serre relations}
\begin{gather*}
  E_iE_j = E_jE_i,\qquad F_iF_j=F_jF_i,\quad |i-j|\geq2,\\
  E_i^2E_{i\pm1} - (q+q^{-1})E_iE_{i\pm1}E_i + E_{i\pm1}E_i^2 = 0,\\
  F_i^2F_{i\pm1} - (q+q^{-1})F_iF_{i\pm1}F_i + F_{i\pm1}F_i^2 = 0.
\end{gather*}
The algebra $U_q(\fsl_{n+1})$ is a Hopf algebra with structure maps given by
\begin{gather*}
  \Delta E_i = E_i\otimes K_i + 1\otimes E_i,\quad
  \Delta F_i = F_i\otimes 1 + K^{-1}_i \otimes F_i,\quad
  \Delta K^{\pm}_i = K^{\pm}_i \otimes K^{\pm}_i,\\
  S(E_i) = - E_iK^{-1}_i,\quad
  S(F_i) = - K_iF_i,\quad
  S(K^{\pm}_i) = K^{\mp}_i,\\
  \eps(K_i) = 0,\quad \eps(E_i)=\eps(F_i)=0.
\end{gather*}
For $q\in\mathbb{R}$, a~Hopf $\ast$-algebra stucture, called the \emph{compact real form} of~$U_q(\fsl_{n+1})$, is defined by
\[
  E_i^\ast := K_iF_i,\qquad
  F_i^\ast := E_iK_i^{-1},\qquad
  K_i^\ast := K_i.
\]
and denoted by~$U_q(\fsu_{n+1})$.

Let $V$ be a~finite-dimensional left $U_q(\fsl_{n+1})$-module, $v\in V$, $f\in V^\ast$. Consider the function
$c^V_{f,v}\colon U_q(\fsl_{n+1})\to \mathbb{C}$ defined by $c^V_{f,v} := f(X v)$. The \emph{coordinate
  space} is the subspace
\[
  C(V) := \Span_{\mathbb{C}}(c^V_{f,v}\mid v\in V, f\in V^\ast) \subseteq U_q(\fsl_{n+1})^\circ,
\]
where $U_q(\fsl_{n+1})^\circ$ denote the dual Hopf algebra of $U_q(\fsl_{n+1})$.
A $U_q(\fsl_{n+1})$-bimodule structure on~$C(V)$ is given by
\[
  (Y c^{V}_{f,v} Z) (X) := f(YXZv) = c^V_{f Z, Yv}(X)
\]
Let $\cP_{+}$ be the set of dominant integral weights of~$\fsl_{n+1}$. For $\lambda\in\cP_+$, denote
by~$V_\lambda$ the irreducible $U_q(\fsl_{n+1})$-module with the highest weight~$\lambda$.
It is easily checked, that the subspace
\[
  O_q(SU_{n+1}) := \bigoplus_{\mu\in\cP_{+}} C(V_\mu)
\]
is a Hopf $\ast$-subalgebra of~$U_q(\fsu_{n+1})^\circ$.

The quantum Levi subalgebra corresponding to the quantum projective space is defined by
\begin{align*}
  U_q(\frak{l}) := \big< K_i, E_j, F_j \,|\, i = 1, \ldots, n+1; j = 2, \dots, n+1 \big> \subseteq U_q(\frak{sl}_{n+1}).
\end{align*}
The Hopf $\ast$-algebra embedding $\iota\colon U_q(\fl)\to U_q(\fsu_{n+1})$ induces a~dual Hopf
algebra map $\iota^\circ\colon U_q(\fsu_{n+1})^\circ \to U_q(\fl)^\circ$.
By construction $\cO_q(SU_{n+1}) \subset U_q(\fsu_{n+1})^\circ$, so we can define the restriction Hopf algebra
map
\[
  \pi := \left.\iota^\circ\right|_{\cO_q(SU_{n+1})}\colon \cO_q(SU_{n+1})\to U_q(\fl_S)^\circ.
\]
Motivated by the classical situation, we denote  $\cO_q(U_n) := \pi\left(\cO_q(SU_{n+1})\right)$
(see~\cite{MEYER, MMF1}). Clearly, $\cO_q(U_n)$ is a~Hopf $\ast$-algebra.
The quantum homogeneous space
\[
  \cO_q(\mathbb{CP}^n) := \cO_q(SU_{n+1})^{\co\cO_q(U_{n})}
\]
corresponding to the surjective Hopf $\ast$-algebra map~$\pi$ is called \emph{quanutm projective space}.

\subsection{Line Bundles over Quantum Projective Space}
The quantum homogeneous space
$\cO_q(SU_{n+1}/SU_n)$
is the invariant subspace of~$\cO_q(SU_n)$ with respect to action of the Hopf subalgebra 
\begin{align*}
  U_q(\fsl_{n}) := \big< K_i, E_i, F_i \,|\, i = 2, \dots, n+1 \big> \subseteq U_q(\frak{sl}_{n+1}).
\end{align*} 
In this special case, the quantum space is usually denoted by $\cO_q(S^{2n+1})$ and called the
\emph{$(2n+1)$-dimensional quantum sphere}.

Since every finite-dimensional representation of~$\frak{sl}_{n+1}$ is contained in some tensor power of the
first fundamental representation~$V_{\varpi_1}$ of~$\frak{sl}_{n+1}$, the matrix coefficients of~$V_{\varpi_1}$
generate~$\cO_q(SU_{n+1})$ as an~algebra. In particular, we can choose a~weight basis~$\{v_j\}_{j=1}^{n+1}$
of~$V_{\varpi_1}$ such that the matrix coefficients $u^i_j := c^{\varpi_1}_{f_i,v_j}$, for $i,j = 1, \dots, n+1$,
coincide with the well-known FRT-presentation of~$\cO_q(SU_{n+1})$, see~\cite{FRT}
or~\cite[\textsection9]{KSLeabh} for details.

With respect to this presentation, the quantum sphere $\cO_q(S^{2n+1})$ is generated as an~algebra by the elements 
\[
  z_i := u^i_1, \qquad \text{and} \qquad  \overline{z}_i := S(u^1_i), \qquad \text{for } i = 1, \dots, n.
\]
The central element~$Z$ in~$U_q(\fl)$, see the discussion in~\cite[\S4.4]{Fano}, is explicitly given by
\begin{equation*}
  Z = K_1^n K_2^{n-1}\cdot \ldots \cdot K_n,
\end{equation*}
In terms of the $\mathbb{Z}$-grading induced by the action of~$Z$, 
\[
  \cO_q(S^{2n+1}) = \bigoplus_{k \in \mathbb{Z}} \cE_k,
\]
the elements~$z_i$ have degree~$1$, while the elements~$\bar z_j$ have degree~$-1$. Thus, for every $k \in \mathbb{N}_0$, as objects
\[
  \cE_k, \, \cE_{-k} \in \modz{\cO_q(SU_{n+1})}{\cO_q(\mathbb{CP}^n)}
\]
the line bundle~$\cE_k$ is generated by the element~$z^k_1$, while~$\cE_{-k}$ is generated by the element~$\overline{z}^k_1$.
  
Finally, for any $k \in \mathbb{N}$, we find it convenient to consider the element $v_{\pm k} \in \Phi(\cE_{\pm k})$, uniquely defined by 
\begin{align} \label{eqn:TAKonLINES}
  \unit(e) = e \otimes v_{\pm k}, & & \textrm{ for all } e \in \cE_{\pm k}.
\end{align} 
As is readily confirmed, $v_k = [z_1^k]$, for all $k \in \mathbb{N}$.

The irreducible quantum flag manifolds (and, in particular, quantum projective spaces) are distinguished by
the existence of an~essentially unique $q$-deformation of their classical de Rham complexes. The existence of
such a~canonical deformation is one of the most important results in the noncommutative geometry of quantum
groups, establishing it as a solid base from which to investigate more general classes of quantum spaces. The
following theorem is a direct consequence of results established in~\cite{HK}, \cite{HKdR}.

\begin{thm}\label{thm:HKClass}
  Over any quantum projective space~$\cO_q(\mathbb{CP}^n)$,
  there exists a unique finite-dimensional left $\cO_q(SU_n)$-covariant differential $*$-calculus
  \[
    \Omega^{\bullet}_q(\mathbb{CP}^n) \in \modz{\cO_q(SU_n)}{\cO_q(\mathbb{CP}^n)},
  \]
  of {classical dimension}, that is to say, satisfying
  \begin{align*}
    \dim \Phi\!\left(\Omega^{k}_q(\mathbb{CP}^n)\right) = \binom{2M}{k}, & & \text{ for all \,} k = 0, \dots, 2 M,
  \end{align*}
  where $M$ is the complex dimension of the corresponding classical manifold~$\mathbb{CP}^n$.
\end{thm}

The calculus $\Omega^{\bullet}_q(\mathbb{CP}^n)$, which we refer to as the \emph{Heckenberger--Kolb calculus} of
$\cO_q(\mathbb{CP}^n)$, has many remarkable properties. We recall here only the existence of a unique covariant
complex structure, following from the results of~\cite{HK}, \cite{HKdR}, and~\cite{MMF2}.

\subsection{A Quantum Principal Bundle Presentation of the Chern Connection of $\cO_q(\mathbb{CP}^n)$}

In this subsection we recall the quantum principal bundle description of the Heckenberger--Kolb calculus introduced in~\cite{MMF1}.
The constituent calculus~$\Omega^1_q(SU_{n+1})$ of the quantum principal bundle was originally constructed as
a~distinguished quotient of the standard bicovariant calculus on~$\cO_q(SU_{n+1})$; see~\cite{JurcoCalculi,MajidBocClass}.
Here we confine ourselves to those properties of the calculus which are relevant to our calculations below,
and refer the interested reader to~\cite[\textsection4]{MMF1}.
We stress that the calculus is far from being a~natural $q$-deformation of the space of differential 1-forms of~$SU_{n+1}$,
instead it should be considered as a~convenient tool for performing explicit calculations.  

The calculus is left $\cO_q(SU_{n+1})$-covariant, right $\cO_q(U_n)$-covariant, and  restricts  to the Heckenberger--Kolb calculus $\Omega^1_q(\mathbb{CP}^n)$. Thus it gives us a quantum principal bundle presentation of $\Omega^1_q(\mathbb{CP}^n)$, with associated short exact sequence
\begin{align*} 
0 \to \cO_q(SU_{n+1})\Om^1_q(\mathbb{CP}^n)\cO_q(SU_{n+1}) {\buildrel \iota \over \longrightarrow} \, \Om^1_q(SU_{n+1}) {\buildrel {~\overline{\mathrm{can}}~~}\over \longrightarrow} \cO_q(SU_{n+1}) \oby \Lambda^1_{\cO_q(U_n)} \to 0.
\end{align*}

Since the calculus is left~$\cO_q(SU_{n+1})$-covariant, it is an~$\cO_q(SU_{n+1})$-Hopf module. A~basis of $F(\Omega^1_q(SU_{n+1}))$ is given by
\begin{align*}
  e^+_i := [\exd u^{i+1}_1],  & &  e^0 := [\exd u^1_1], & & e^-_i := [\exd u^1_{i+1}], & & \textrm{ for } i = 1, \dots, n.
\end{align*}
Moreover, $[u^i_j] = 0$, if both $i,j \neq 1$. Let us now denote 
\begin{align*}
  \Lambda^{(1,0)} := \mathrm{span}_{\mathbb{C}}\{e^+_i \,|\, i=1, \dots, n\}, & & \Lambda^{(0,1)} := \mathrm{span}_{\mathbb{C}}\{e^-_i \,|\, i=1, \dots, n\}.
\end{align*}
The space $\Lambda^{(1,0)}\oplus \Lambda^{(0,1)}$ is a~right $\cO_q(SU_{n+1})$-sub-module of $\Phi_{\cO_q(SU_{n+1})}\!\left(\Omega^1_q(SU_{n+1})\right)$.
Explicitly, its $\cO_q(SU_{n+1})$-sub-module structure is given by 
\begin{align} \label{eqn:Lambdarightaction}
  e^{\pm}_i \tl u^k_k = q^{\delta_{i+1,k}+\delta_{1k} -2/(n+1)} e^{\pm}_i, & & e^{\pm}_i \tl u^k_l = 0, \textrm{ ~~ for all } k \neq l.
\end{align}
It is important to note that the subspace $\mathbb{C}e^0$ is \emph{not} an~$\cO_q(SU_{n+1})$-sub-module of~$F\!\left(\cO_q(SU_{n+1})\right)$,
nor is it even an~$\cO_q(S^{2n-1})$-sub-module.
However, as shown in~\cite[Proposition 6.2]{MMF1},
it \emph{is} a~sub-module over $\mathbb{C}\langle z_1 \rangle$, the $*$-sub-algebra of~$\cO_q(S^{2n+1})$ generated by~$z_1$.

It follows from the results of~\cite[\textsection5]{MMF1} that
\[
  F\left(\cO_q(SU_{n+1})\Omega^1_q(\mathbb{CP}^n)\cO_q(SU_{n+1})\right) =   \Lambda^{(1,0)} \oplus  \Lambda^{(0,1)}.
\]
Moreover, a~decomposition of right~$\cO_q(U_n)$-comodules is given by
\begin{align*}
  F\left(\Omega^1_q(SU_{n+1})\right) = \Lambda^{(1,0)} \oplus \Lambda^{(0,1)} \oplus  \mathbb{C} e^0.
\end{align*}
Thus we have a~left $\cO_q(SU_{n+1})$-covariant strong principal connection~$\Pi$, uniquely defined by 
\begin{align*} 
  F(\Pi): F\!\left(\Omega^1_q(SU_{n+1})\right) \to \mathbb{C}e^0.
\end{align*}

For an~arbitrary covariant vector bundle~$\cF$, let us now look at the associated connection 
\[
  \nabla:\cF \to \Omega^{1}_q(\mathbb{CP}^n) \otimes_{\cO_q(\mathbb{CP}^n)} \cF
\] 
associated to $\Pi$. The linear map
\[
  \adel_{\cF} := (\proj_{\Omega^{(0,1)}} \otimes \id) \circ \nabla: \cF \to \Omega^{(0,1)}  \otimes_{\cO_q(\mathbb{CP}^n)}  \cF
\]
is a~$(0,1)$-connection.
Moreover, we have an~analogously defined $(1,0)$-connection for~$\cF$,
which we denote by~$\del_{\cF}$. Consider next  the obvious linear projections
\begin{align} \label{eqn:projections}
  \Pi^{(1,0)}:F\!\left(\Omega^1_q(SU_{n+1})\right) \to \Lambda^{(1,0)}, & & \Pi^{(0,1)}: F\!\left(\Omega^1_q(SU_{n+1})\right) \to \Lambda^{(0,1)}.
\end{align}
In terms of these operators, we have the following useful formulae:
\begin{align*}
  \del_{\cF} =  j^{-1} \circ ((\Pi^{(1,0)} \circ \exd) \otimes \mathrm{id}) \circ \unit,\\
  \adel_{\cF} = j^{-1} \circ ((\Pi^{(0,1)} \circ \exd) \otimes \mathrm{id}) \circ \unit.
\end{align*}

For the special case of the covariant line bundles, it follows from the uniqueness of $(0,1)$-connections,
presented in~\cite[Theorem 4.5]{HolVBs}, that $\adel_{\cE_k}$ is equal to 
the holomorphic structure of $\cE_k$, justifying the choice of notation. We have an analogous result
for the $(1,0)$-connection~$\del_{\cE_k}$.
Thus $\nabla = \del_{\cE_k} + \adel_{\cE_k}$ is equal to the Chern connection of~$\cE_k$.

\subsection{Chern Curvature of the Positive Line Bundles of $\cO_q(\mathbb{CP}^n)$} \label{section:CPNCurvature}

In this subsection we explicitly calculate the curvature of the positive line bundles over quantum projective space.
We begin with the following technical lemma.

\begin{lem} \label{lem:qLieb}
  It holds that, for all $k \in \mathbb{N}$, 
  \[
    \Pi^{(1,0)} \circ \exd(z^k_1) = (k)_{q^{2/(n+1)}}\!\left(\Pi^{(1,0)} \circ \exd(z_1)\right)\!z_1^{k-1}.
  \]
\end{lem}
\begin{proof}
  We will prove the formula using induction. For $k=1$, the formula is trivially satisfied.
  For $k=2$, we see that
  \begin{align*}
    \unit\!\left(\Pi^{(1,0)} \circ \exd(z_1)z_1\right) = {}& \, \unit\!\left(\Pi^{(1,0)} \circ \exd(u^1_1)\right)\!u^1_1 \\
    {}={} &  \left(\sum_{a=2}^n u^1_a \otimes [\exd(u^a_1)]\right)\!u^1_1 \\
    {}={} & \sum_{b=1}^n \sum_{a=2}^n u^1_au^b_1 \otimes [\exd(u^a_1)u^1_b].
  \end{align*}
  Recalling  the identities in~\eqref{eqn:Lambdarightaction} and the definition of~$\Pi^{(1,0)}$ given in~\eqref{eqn:projections}, we see that
  \begin{align*}
    \sum_{b=1}^n \sum_{a=2}^n u^1_au^b_1 \otimes [\exd(u^a_1)u^1_b] = \sum_{a=2}^n u^1_au^1_1 \otimes [\exd(u^a_1)u^1_1] = & ~~  q^{1 -\frac{2}{n+1}} \sum_{a=2}^n  u^1_au^1_1 \otimes [\exd(u^a_1)].
  \end{align*}
  The commutation relations of~$\cO_q(SU_{n+1})$ tell us that $u^1_au^1_1 = q^{-1}u^1_1 u^1_a$
  (see, for example, \cite[\textsection1]{FRT} or~\cite[\textsection9.2]{KSLeabh} for details). Thus 
  \begin{align*}
    q^{1 -\frac{2}{n+1}} \sum_{a=2}^n  u^1_au^1_1 \otimes [\exd u^a_1] = & ~ q^{-\frac{2}{n+1}} u^1_1 \sum_{a=2}^n u^1_a \otimes [\exd u^a_1] \\
    = & ~ q^{-\frac{2}{n+1}}  \unit\left( z_1\, \Pi^{(1,0)} \circ \exd(z_1)\right)\!.
  \end{align*}
  Hence we see that $z_1\! \left(\Pi^{(1,0)} \circ \exd(z_1)\right) = q^{\frac{2}{n+1}} \left(\Pi^{(1,0)} \circ \exd(z_1)\right)\!z_1$.

  Let us now assume that the formula holds for some general~$k$. By the Leibniz rule
  \begin{align*}
    \Pi^{(1,0)} \circ \exd(z^{k+1}_1) = & \,   \Pi^{(1,0)}\!\left((\exd z_1^{k-1})z_1 + z_1^{k-1} \exd z_1\right).
  \end{align*}
  Since $\Pi^{(1,0)}$ is a~left $\cO_q(SU_{n+1})$-module map, it must hold that 
  \begin{align*}
    \Pi^{(1,0)}\!\left((\exd z_1^{k-1})z_1 + z_1^{k-1} \exd z_1\right) = & \, \Pi^{(1,0)}\!\left(\exd z_1^{k-1}z_1\right)  + z_1^{k-1} \Pi^{(1,0)}\left(\exd z_1\right).
  \end{align*}
  Moreover, since $\mathbb{C}e^0$ is a~$\mathbb{C}\langle z_1 \rangle$-sub-module of $F\!\left(\Omega^1_q(SU_{n+1})\right)$,
  the projection~$\Pi^{(1,0)}$ must be a~right $\mathbb{C}\langle z_1 \rangle$-module map.
  Thus we see that 
  \begin{align*}
    \Pi^{(1,0)}\!\left(\exd z_1^{k-1}z_1\right)  + z_1^{k-1} \Pi^{(1,0)}\left(\exd z_1\right) = \Pi^{(1,0)}\!\left(\exd z_1^{k-1}\right)z_1  + z_1^{k-1} \Pi^{(1,0)}\left(\exd z_1\right).
  \end{align*}
 
  Using our inductive assumption, we can reduce this expression to 
  \begin{align*}
    (k-1)_{q^{2/(n+1)}} \Pi^{(1,0)}\!\left(\exd z_1\right) z^{k-1}_1 + q^{2(k-1)/(n+1)} \Pi^{(1,0)}(\exd z_1) z^{k-1}_1.
  \end{align*}
  By the definition of the quantum integer, this in turn reduces to 
  \begin{align*}
    (k)_{q^{2/(n+1)}} \Pi^{(1,0)}\!\left(\exd z_1\right) z^{k-1}_1.
  \end{align*}
  The claimed formula now follows by induction.
\end{proof}

\begin{thm} \label{tm:curvk}
  For any positive line bundle~$\cE_k$ over quantum projective space~$\mathcal{O}_q(\mathbb{CP}^n)$, it holds that 
  \begin{align*}
    \nabla^2(e) =  -(k)_{q^{-2/(n+1)}} \mathbf{i} \kappa \otimes e, & & \textrm{ for all } e \in \mathcal{E}_k,
  \end{align*}
  where we have chosen the unique K\"ahler form~$\kappa$ satisfying
  \begin{align} \label{eqn:firstCPNChern}
    \nabla^2(e) = - \mathbf{i} \kappa \otimes e, & & \text{for all }   e \in \cE_1.
  \end{align}
\end{thm}
\begin{proof}
  We are free to calculate the curvature of~$\cE_k$ by letting~$\nabla^2$ act on any non-zero element of~$\cE_k$.
  The element~$z_1^k$ presents itself as a~convenient choice since $\adel_{\mathcal{E}_k}(z_1^k) = 0$, as proved in~\cite[\S???]{MMF1}.
  In particular, it holds that 
  \[
    \nabla^2(z^k_1) = (\adel_{\cE_k} \circ \del_{\cE_k} + \del_{\cE_k} \circ \adel_{\cE_k})(z^k_1) = \adel_{\cE_k} \circ \del_{\cE_k}(z^k_1).
  \]

  For convenience, let us now denote $\alpha := q^{-2/(n+1)}$.
  From the quantum principal bundle presentation of~$\del_{\cE_k}$ given in the previous subsection,
  together with Lemma~\ref{lem:qLieb}, we see that 
  \begin{align*}
    \del_{\cE_k}(z_1^k) ={}  & j^{-1}\!\left(\Pi^{(1,0)} \circ \exd(z^k_1) \otimes v_k\right)\\
                     {} ={} & (k)_{\alpha} \, j^{-1}\!\left( \left(\Pi^{(1,0)} \circ \exd(z^k_1)\right) \!z^{k-1}_1 \otimes v_k\right)\!,
  \end{align*}
  where in the first identity we have used~\eqref{eqn:TAKonLINES}.
  We now present this expression as an~element in $\Omega^{(1,0)} \otimes_B (A \square_H \Phi(\mathcal{E}_k))$: 
  \begin{align*}
    (k)_{\alpha} j^{-1}\!\left(\left(\Pi^{(1,0)}\circ \exd(z_1)\right)\!z^{k-1}_1 \otimes v_k\right) = {}
    & \sum_{i=1}^{n+1} (k)_{\alpha} j^{-1}  \left(\Pi^{(1,0)}\circ \exd(u^1_1)S(u^1_i)u^i_1 z^{k-1}_1 \otimes v_k\right)\\
    {} = {} & \sum_{i=1}^{n+1} (k)_{\alpha} \del(u^1_1S(u^1_i)) \otimes (z_i z_1^{k-1} \otimes v_k).
  \end{align*}

  Acting on this element by~$\adel_{\mathcal{E}_k}$, and recalling that~$\adel z_i = 0$, for all $i = 1,\dots, n$,  gives us the identity
  \begin{align*}
    \sum_{i=1}^{n+1}\adel_{\mathcal{E}_k}\!\Big((k)_{\alpha} \del(u^1_1S(u^1_i)) \otimes (z_i z_1^{k-1} \otimes v_k)\Big) = {}
    & \sum_{i=1}^{n+1} (k)_{\alpha}  \adel \del(u^1_1S(u^1_i)) \otimes (z_i z_1^{k-1} \otimes v_k).
  \end{align*}
  Operating by $\id \otimes \unit^{-1}$ produces the expression
  \begin{align*}
    (k)_{\alpha} \sum_{i=1}^{n+1}\adel \del(u^1_1S(u^1_i)) \otimes z_i z_1^{k-1}  = {} & (k)_{\alpha}  \nabla(u^1_1) z^{k-1}_1,
  \end{align*}
  where the right multiplication of~$\nabla(u^1_1)$ by~$z^{k-1}_1$ is defined with respect to
  the canonical embeddings of $\Omega^{(1,1)} \otimes_B \cE_1$ and $\Omega^{(1,1)} \otimes_B \cE_k$
  into $\Omega^{(1,1)} \otimes_B \cO_q(S^{2n-1})$.
  Finally, recalling that we have chosen a~scaling of our K\"ahler form to satisfy~\eqref{eqn:firstCPNChern}, we have
  \begin{align*}
    \nabla^2(z^k_1) =  (k)_{\alpha} (-\mathbf{i} \, \kappa \otimes z_1)z^{k-1}_1 
    =  - (k)_{\alpha} \mathbf{i}  \kappa \otimes z_1^k,
  \end{align*}
  which gives us the claimed identity.
\end{proof}

\begin{remark}
  It is worth noting that the Chern curvature is clearly independent of any quantum principal bundle
  presentation of the calculus~$\Omega^1_q(\mathbb{CP}^{n})$.
  However, the quantum bundle presentation allows us to calculate the curvature in a~systematic manner, and
  provides us with concrete insight into why the curvature undergoes a $q$-integer deformation.
\end{remark}

\appendix

\section{Some Categorical  Equivalences} \label{app:A}

In this appendix we present a~number of categorical equivalences, all ultimately derived from Takeuchi's
equivalence~\cite{Tak}.
(We note that similar results hold under much weaker assumptions, see~\cite{Skryabin2007}.)
These equivalences play a~prominent role in the paper, giving us a~formal framework
in which to understand covariant differential calculi as noncommutative homogeneous vector bundles.

\subsection{Takeuchi's Bimodule  Equivalence} \label{app:TAK}

Let $A$ and~$H$ be Hopf algebras, and $B = A^{\co(H)}$ the quantum homogeneous space associated to
a~surjective Hopf algebra map $\pi:A \to H$. We define $\qMod{A}{B}{}{B}$ to be the category whose objects are 
left \mbox{$A$-comodules} \mbox{$\DEL_L:\mathcal{F} \to A \otimes \mathcal{F}$}, endowed with a~$B$-bimodule
structure, such that
\begin{align} \label{eqn:TakCompt}
  \DEL_L(bfc) = \Delta_L(b)\DEL_L(f)\Delta_L(c),  & & \text{ for all } f \in \mathcal{F}, b,c \in B,
\end{align}
and whose morphisms are left $A$-comodule, $B$-bimodule, maps.
Let $\lMod{H}{}$ denote the category whose objects are left $H$-comodules, and whose morphisms are left $H$-comodule maps.

If $\cF \in \qMod{A}{B}{}{B}$, and $B^+ := B \cap \ker(\e:A \to \mathbb{C})$, then~$\cF/(B^+\cF)$ becomes
an~object in~$\qMod{H}{}{}{B}$ with the obvious right $B$-action, and left $H$-coaction given by 
\begin{align} \label{comodstruc0}
  \DEL_L[f] = \pi(f_{(-1)}) \oby [f_{(0)}], & & \text{ for } f \in \cF,
\end{align}
where~$[f]$ denotes the coset of~$f$ in  $\cF/(B^+\cF)$. A functor 
\begin{align} \label{eqn:functorPHI}
  \Phi: \qMod{A}{B}{}{B} \to \qMod{H}{}{}{B}
\end{align}
is now defined as follows:  
$\Phi(\cF) := \cF/(B^+\cF)$, and if $g : \cF \to {\mathcal D}$ is a~morphism in~$\qMod{A}{B}{}{B}$,
then $\Phi(g):\Phi(\cF) \to \Phi({\mathcal D})$ is the map uniquely defined by $\Phi(g)[f] := [g(f)]$.

If $V \in \qMod{H}{}{}{B}$ with coaction $\Delta_L : V \to H \otimes V$, then the {\em cotensor product} of~$A$ and~$V$ is defined by
\begin{align*}
  A \coby V := \ker(\DEL_R \oby \id - \id \oby \DEL_L: A\oby V \to A \oby H \oby V),
\end{align*}
where $\Delta_R : A \to A \otimes H$ denotes the homogenous right $H$-coaction on~$A$.
The cotensor product becomes an~object in~$\qMod{A}{B}{}{B}$ by defining a~left $B$-bimodule structure,
and left $A$-comodule structure, on the first tensor factor in the obvious way, and defining a~right $B$-module structure by
\begin{align*}
  \left(\sum a_i \otimes v_i\right)b := \sum a_ib_{(1)} \otimes \left(v_i \tl b_{(2)}\right), 
\end{align*}
for any $b \in B$, and any $\sum a_i \otimes v_i  \in  A \square_H V$.
A functor 
\[
  \Psi: \qMod{H}{}{}{B} \to \qMod{A}{B}{}{B}
\] 
is now defined as follows: 
\(
\Psi(V) := A \square_H V,
\)
and if~$\gamma$ is a~morphism in~$\qMod{H}{}{}{B}$, then $\Psi(\gamma) := \id \oby \gamma$. 

For a~quantum homogeneous space $B = A^{\co(H)}$,
the algebra~$A$ is said to be {\em faithfully flat} as a~right $B$-module if the functor
$A \oby_B -: \lMod{}{B} \to \lMod{}{\mathbb{C}}$,
from the category of left $B$-modules to the category of complex vector spaces,
preserves and reflects exact sequences.
As shown in~\cite[Corollary 3.4.5]{Chirva},
for any coideal \mbox{$*$-subalgebra} of a~CQGA faithful flatness is automatic.
For example, $\cO_q(G)$ is faithfully flat as a~left module over any quantum flag manifold $\cO_q(G/L_S)$.
The following equivalence was established in~\cite[Theorem 1]{Tak}.

\begin{thm}[Takeuchi's Equivalence] \label{thm:TakEquiv}
  Let $B = A^{\co(H)}$ be a~quantum homogeneous space such that~$A$ is faithfully flat as a~right $B$-module.
  An~adjoint equivalence of categories between  $\qMod{A}{B}{}{B}$ and~$\qMod{H}{}{}{B}$
  is given by the functors~$\Phi$ and~$\Psi$ and unit, and counit,  natural isomorphisms
  \begin{align*}
    \unit: \cF \to \Psi \circ \Phi(\cF), & & f \mto f_{(-1)} \oby [f_{(0)}], \\
    \counit:\Phi \circ  \Psi(V) \to V, \,  & & \Big[\sum_i a^i \oby v^i\Big] \mto \sum_{i} \e(a^i)v^i.
  \end{align*} 
\end{thm}

As observed in~\cite[Corollary 2.7]{MMF2}, the inverse of the unit~$\unit$ of the equivalence admits a useful explicit description: 
\begin{align} \label{eqn:unitinverse} 
  \unit^{-1}\!\left(\sum_i f_i \otimes [g_i] \right) = \sum_i f_iS\big((g_i)_{(-1)}\big)(g_i)_{(0)}.
\end{align}

\subsection{The Fundamental Theorem of Two-Sided Hopf Modules}
\label{app:fun}

In this subsection we consider a~special case of Takeuchi's equivalence, namely the fundamental theorem of
two-sided Hopf modules. (This equivalence was originally considered in~\cite[Theorem 5.7]{Schauenburg} using
a~parallel but equivalent formulation, see also~\cite{Saraccoa19}.)
For a~Hopf algebra $A$, the counit $\e:A \to \mathbb{C}$ is a~Hopf algebra map.
The associated quantum homogeneous space is given by~$A = A^{\co(\mathbb{C})}$,
the category~$\qMod{A}{B}{}{B}$ specialises to~$\qMod{A}{A}{}{A}$,
and  the category $\qMod{H}{}{}{B}$ reduces to the category of right $A$-modules $\rMod{}{A}$.
In this special case we find it useful to denote the functor~$\Phi$  as
\begin{align*}
  F: \qMod{A}{A}{}{A}  \to \rMod{}{A}, & & \cF \mapsto \cF/A^+\cF,
\end{align*}
Moreover, since the cotensor product over~$\mathbb{C}$ is just the usual tensor product~$\otimes$,
we see that the functor~$\Psi$ reduces to 
\begin{align*}
  A \otimes -: \rMod{}{A}  \to \qMod{A}{A}{}{A}, & & V \mapsto A \otimes V.
\end{align*}
Since faithful flatness is trivially satisfied in this case, we have the following corollary of Takeuchi's equivalence.

\begin{thm}[Fundamental Theorem of Two-Sided Hopf Modules] \label{thm:FunThm2Side}
  An adjoint equivalence of  categories between~$\qMod{A}{A}{}{A}$ and  $\rMod{}{A}$
  is given by the functors~$F$ and $A \otimes -$, and the unit, and counit, natural isomorphisms
  \begin{align*}
    \unit:& ~ \cF \to A \otimes F(\cF),  & ~~ f \mto f_{(-1)} \oby [f_{(0)}],\\
    \counit:&~ F(A \otimes V) \to V,  ~  & \left[ a \oby v \right] \mto  \e(a)v. ~~ 
  \end{align*} 
\end{thm}

\subsection{Some Monoidal Equivalences}

In this subsection we recall two monoidal equivalences induced by Takeuchi's equivalence.
Denote by~$\Modz{A}{B}$ the full sub-category of~$\qMod{A}{B}{}{B}$ whose objects~$\cF$ satisfy the identity 
$
\mathcal{F}B^+ = B^+\mathcal{F}. 
$ 
Consider also the full sub-category of $\qMod{H}{}{}{B}$ consisting of those objects endowed with the trivial
right $B$-action, which is to say, those objects~$V$ for which $v \tl b = \e(b)v$, for all $v \in V$,
and $b \in B$.
This category is clearly isomorphic to~$\lMod{H}{}$, the category of left $H$-comodules, and as such,
Takeuchi's equivalence induces an equivalence between~$\Modz{A}{B}$ and~$\lMod{H}{}$,  for details see~\cite[Lemma 2.8]{MMF2}.

For $\mathcal{F},\mathcal{D}$ two objects in~$\Modz{A}{B}$, we denote by~$\mathcal{F} \oby_B \mathcal{D}$
the usual bimodule tensor product endowed with the standard left $A$-comodule structure.
This gives $\qMod{A}{B}{}{B}$  the structure of a~monoidal category.
If $\mathcal{F},\mathcal{D}$ are contained in the subcategory~$\Modz{A}{B}$,
then it is easily checked that $\mathcal{F} \oby_B \mathcal{D}$ is again an~object in~$\Modz{A}{B}$.
Thus $\Modz{A}{B}$ is a~monoidal subcategory of~$\qMod{A}{B}{}{B}$.
With respect to the usual tensor product of comodules in~$\lMod{H}{}$,
Takeuchi's equivalence is given the structure of a~monoidal equivalence
(see~\cite[\textsection4]{MMF2} for details) by the morphisms
\begin{align*}
  \mu_{\mathcal{F}, \mathcal{D}}: \Phi(\mathcal{F}) \oby \Phi(\mathcal{D}) \to \Phi(\mathcal{F} \oby_B \mathcal{D}), & & [f] \oby [d] \mto [f \oby d], &  \text{ ~~~ for any  } \mathcal{F},\mathcal{D} \in \Modz{A}{B}.
\end{align*}
This monoidal equivalence will be tacitly assumed throughout the paper, along with the implied monoid
structure on~$\Phi(\N)$, for any monoid object $\N \in \Modz{A}{B}$. 

Consider now the category $\lMod{A}{B}$, whose objects are left $A$-comodules, and left $B$-modules,
satisfying the obvious analogue of~\eqref{eqn:TakCompt}, and whose morphisms are left $A$-comodule,
right $B$-module maps. We can endow any object $\cF \in \lMod{A}{B}$ with a~right $B$-action uniquely defined  by
\[
  f \triangleleft b := f_{(-2)}bS(f_{(-1)})f_{(0)}.
\]
Since $e_{(-2)}bS(e_{(-1)})e_{(0)} \in B^+\cF$, for all $b \in B^+$, this new right module structure satisfies
the defining conditions of~$\Modz{A}{B}$, giving us an~obvious equivalence between~$\lMod{A}{B}$ and~$\Modz{A}{B}$.
In particular, we see that any left $A$-comodule, left $B$-module map between two
objects in~$\Modz{A}{B}$ is automatically a~morphism.
(We should note that the implied equivalence
between~$\lMod{A}{B}$ and~$\lMod{H}{}$ is the original form of Takeuchi's equivalence~\cite{Tak}, the
bimodule form presented above being an easy consequence.)

Next we examine~$\modz{A}{B}$ the full sub-category of $\qMod{A}{B}{}{B}$ whose objects~$\cF$ are finitely
generated as left $B$-modules and $\qmod{H}{}{}{B}$, the full sub-category of~$\qMod{H}{}{}{B}$ whose
objects are finite-dimensional as complex vector spaces. As established in~\cite[Corollary~2.5]{MMF3}
Takeuchi's equivalence induces an~equivalence between these two sub-categories.
We define the {\em dimension} of an~object $\mathcal{F} \in ^A_B \rmod{}{B}$ to be the dimension
of~$\Phi(\mathcal{F})$ as a~vector space.

\section{Quantum Integers} \label{app:quantumintegers}

Quantum integers are ubiquitous in the study of quantum groups.
For this paper in particular, they arise in the defining relations of the Drinfeld--Jimbo quantum groups,
and in the calculation of the curvature of the positive line bundles over quantum projective
space~$\mathcal{O}_q(\mathbb{CP}^n)$.
In each case we use different but related formulations for quantum integers.
Thus we take care here to clarify our choice of conventions.
We begin with the version of quantum integer used in the definition of the Drinfeld--Jimbo quantum groups.
For $q \in \bC$, the {\em quantum integer} $[m]_q$  is  the complex number 
\[
  [m]_q := q^{-m+1} + q^{-m+3} + \cdots + q^{m-3} + q^{m-1}.
\]
Note that when $q \notin \{-1,0,1\}$, we have the identity
\[
  [m]_q  = \frac{q^m-q^{-m}}{q-q^{-1}}.
\]
We next recall the definition of the quantum binomials, which arise in the quantum Serre relations of the
Drinfeld--Jimbo quantum groups. For any $n \in \mathbb{N}$, we denote
\begin{align*}
  [n]_q! = [n]_q[n-1]_q \cdots [2]_q[1]_q,
\end{align*}
and moreover, we denote $[0]_q! = 1$.
For any non-zero $q \in \mathbb{C}$, and any $n,r \in \mathbb{N}_0$, the associated $q$-binomial coefficient is the complex number
\begin{align*}
  \begin{bmatrix} n \\ r \end{bmatrix}_q := \frac{[n]_q!}{[r]_q! \, [n-r]_q!}.
\end{align*}

By contrast, the form of quantum integer arising in curvature calculations is defined as follows:
For $q \in \bC\setminus\{1\}$, the {\em quantum integer} $(m)_q$  is the complex number 
\begin{align*}
  (m)_q  = \frac{1-q^m}{1-q}.
\end{align*}
When $m > 0$, we have
\begin{align*}
  (m)_q := 1 + q + q^2 + \cdots + q^{m-1}.
\end{align*}

The definition of quantum binomial also makes sense for this version of quantum integer, although we will not
use it in this paper. Finally, it is instructive to note that the two conventions are related by the identity
\begin{align*}
  [m]_q = q^{1-m} (m)_{q^2}.
\end{align*}

\newcommand{\etalchar}[1]{$^{#1}$}
\providecommand{\bysame}{\leavevmode\hbox to3em{\hrulefill}\thinspace}
\providecommand{\MR}{\relax\ifhmode\unskip\space\fi MR }
\providecommand{\MRhref}[2]{%
  \href{http://www.ams.org/mathscinet-getitem?mr=#1}{#2}
}
\providecommand{\href}[2]{#2}


\begin{thebibliography}{DKOSS1}

\bibitem[AFL]{AschieriFioresiLatini2021}
P.~Aschieri, R.~Fioresi, and E.~Latini, \emph{Quantum principal bundles on
  projective bases}, Comm. Math. Phys. \textbf{382} (2021), no.~3, 1691--1724,
  \href{http://arxiv.org/abs/1907.12751}{{\ttfamily arXiv:1907.12751
  [math.QA]}}.

\bibitem[BM1]{BeggsMajid:Leabh}
E.~Beggs and S.~Majid, \emph{Quantum {R}iemannian geometry}, 1 ed., Grundlehren
  der mathematischen Wissenschaften, vol. 355, Springer International
  Publishing, 2019.

\bibitem[BM2]{TBSM1}
T.~Brzezi\'{n}ski and S.~Majid, \emph{Quantum group gauge theory on quantum
  spaces}, Comm. Math. Phys. \textbf{157} (1993), no.~3, 591--638,
  \href{http://arxiv.org/abs/hep-th/9208007}{{\ttfamily arXiv:hep-th/9208007}}.

\bibitem[BM3]{TBSMRevisited}
T.~Brzezi\'{n}ski and S.~Majid, \emph{Quantum differentials and the
  {$q$}-monopole revisited}, Acta Appl. Math. \textbf{54} (1998), no.~2,
  185--232, \href{http://arxiv.org/abs/q-alg/9706021}{{\ttfamily
  arXiv:q-alg/9706021}}.

\bibitem[Brz]{TBGS}
T.~Brzezi\'{n}ski, \emph{Galois structures}, Available at
  \url{https://www.impan.pl/swiat-matematyki/notatki-z-wyklado~/brzezinski_gs.pdf},
  2008.

\bibitem[BS1]{BS}
E.~Beggs and P.~S. Smith, \emph{{Noncommutative complex differential
  geometry}}, J. Geom. Phys. \textbf{72} (2013), 7--33,
  \href{http://arxiv.org/abs/1209.3595}{{\ttfamily arXiv:1209.3595 [math.AG]}}.

\bibitem[BS2]{BrzezinskiSzymanski2021}
T.~Brzezi\'{n}ski and W.~Szyma\'{n}ski, \emph{An algebraic framework for
  noncommutative bundles with homogeneous fibres}, Algebra Number Theory
  \textbf{15} (2021), no.~1, 217--240,
  \href{http://arxiv.org/abs/1911.12075}{{\ttfamily arXiv:1911.12075
  [math-ph]}}.

\bibitem[Chi]{Chirva}
A.~Chirvasitu, \emph{Relative {F}ourier transforms and expectations on coideal
  subalgebras}, J. Algebra \textbf{516} (2018), 271--297,
  \href{http://arxiv.org/abs/1802.03097}{{\ttfamily arXiv:1802.03097
  [math.QA]}}.

\bibitem[CM]{Branimir2021CMP}
B.~\'{C}a\'{c}i\'{c} and B.~Mesland, \emph{Gauge {T}heory on {N}oncommutative
  {R}iemannian {P}rincipal {B}undles}, Comm. Math. Phys. \textbf{388} (2021),
  no.~1, 107--198, \href{http://arxiv.org/abs/1912.04179}{{\ttfamily
  arXiv:1912.04179 [math-ph]}}.

\bibitem[CMO]{BWGrass}
A.~Carotenuto, C.~Mrozinski, and R.~\'O~Buachalla, \emph{A {B}orel--{W}eil
  theorem for the quantum {G}rassmannians}, Doc. Math. (2023),
  \href{http://arxiv.org/abs/1611.07969}{{\ttfamily arXiv:1611.07969
  [math.QA]}}, To appear.

\bibitem[DKOSS1]{HolVBs}
F.~D\'{i}az~Garc\'{i}a, A.~Krutov, R.~\'O~Buachalla, P.~Somberg, and K.~R.
  Strung, \emph{Holomorphic relative {H}opf modules over the irreducible
  quantum flag manifolds}, Lett. Math. Phys. \textbf{111} (2021), no.~1, 24,
  \href{http://arxiv.org/abs/2005.09652}{{\ttfamily arXiv:2005.09652
  [math.QA]}}.

\bibitem[DKOSS2]{Fano}
F.~D\'{i}az~Garc\'{i}a, A.~Krutov, R.~\'O~Buachalla, P.~Somberg, and K.~R.
  Strung, \emph{Positive line modules over the irreducible quantum flag
  manifolds}, Lett. Math. Phys. \textbf{112} (2022), no.~6, 33,
  \href{http://arxiv.org/abs/1912.08802}{{\ttfamily arXiv:1912.08802
  [math.QA]}}, OWP-2020-01.

\bibitem[DL]{Antiselfdual}
F.~D'Andrea and G.~Landi, \emph{Anti-selfdual connections on the quantum
  projective plane: instantons}, Comm. Math. Phys. \textbf{333} (2015), no.~1,
  505--540, \href{http://arxiv.org/abs/1305.1246}{{\ttfamily arXiv:1305.1246
  [math.QA]}}.

\bibitem[DOBS]{DOS1}
B.~Das, R.~\'{O}~Buachalla, and P.~Somberg, \emph{A {D}olbeault-{D}irac
  spectral triple for quantum projective space}, Doc. Math. \textbf{25} (2020),
  1079--1157, \href{http://arxiv.org/abs/1903.07599}{{\ttfamily
  arXiv:1903.07599 [math.QA]}}.

\bibitem[DVM]{DVMichor}
M.~Dubois-Violette and P.~W. Michor, \emph{Connections on central bimodules in
  noncommutative differential geometry}, J. Geom. Phys. \textbf{20} (1996),
  no.~2-3, 218--232, \href{http://arxiv.org/abs/q-alg/9503020}{{\ttfamily
  arXiv:q-alg/9503020 [q-alg]}}.

\bibitem[DMMM1]{DVMadoreMouradBimodule}
M.~Dubois-Violette, J.~Madore, T.~Masson, and J.~Mourad, \emph{Linear
  connections on the quantum plane}, Lett. Math. Phys. \textbf{35} (1995),
  no.~4, 351--358, \href{http://arxiv.org/abs/hep-th/9410199}{{\ttfamily
  arXiv:hep-th/9410199 [hep-th]}}.

\bibitem[DMMM2]{MadoreBimodule}
M.~Dubois-Violette, J.~Madore, T.~Masson, and J.~Mourad, \emph{On curvature in
  noncommutative geometry}, J. Math. Phys. \textbf{37} (1996), no.~8,
  4089--4102, \href{http://arxiv.org/abs/q-alg/9512004}{{\ttfamily
  arXiv:q-alg/9512004 [q-alg]}}.

\bibitem[FRT]{FRT}
L.~D. Faddeev, N.~Y. Reshetikhin, and L.~A. Takhtadzhyan, \emph{Quantization of
  {L}ie groups and {L}ie algebras}, Algebra i Analiz \textbf{1} (1989), no.~1,
  178--206.

\bibitem[Haj]{PHC}
P.~M. Hajac, \emph{A note on first order differential calculus on quantum
  principal bundles}, Czechoslovak J. Phys. \textbf{47} (1997), no.~11,
  1139--1143, Quantum groups and integrable systems, Part~I (Prague, 1997).

\bibitem[HK1]{HK}
I.~Heckenberger and S.~Kolb, \emph{The locally finite part of the dual
  coalgebra of quantized irreducible flag manifolds}, Proc. London Math. Soc.
  (3) \textbf{89} (2004), no.~2, 457--484,
  \href{http://arxiv.org/abs/math/0301244}{{\ttfamily arXiv:math/0301244
  [math.QA]}}.

\bibitem[HK2]{HKdR}
I.~Heckenberger and S.~Kolb, \emph{De {R}ham complex for quantized irreducible
  flag manifolds}, J. Algebra \textbf{305} (2006), no.~2, 704--741,
  \href{http://arxiv.org/abs/math/0307402}{{\ttfamily arXiv:math/0307402
  [math.QA]}}.

\bibitem[Huy]{HUY}
D.~Huybrechts, \emph{Complex geometry: an introduction}, 1 ed., {U}niversitext,
  Springer--Verlag Berlin Heidelberg, 2005.

\bibitem[Jur]{JurcoCalculi}
B.~Jur\v{c}o, \emph{Differential calculus on quantized simple {L}ie groups},
  Lett. Math. Phys. \textbf{22} (1991), no.~3, 177--186.

\bibitem[KLvS]{KLvSPodles}
M.~Khalkhali, G.~Landi, and W.~D. van Suijlekom, \emph{Holomorphic structures
  on the quantum projective line}, Int. Math. Res. Not. IMRN (2011), no.~4,
  851--884, \href{http://arxiv.org/abs/0907.0154}{{\ttfamily arXiv:0907.0154
  [math.QA]}}.

\bibitem[KM]{KoszulMalgrange}
J.-L. Koszul and B.~Malgrange, \emph{Sur certaines structures fibr\'{e}es
  complexes}, Arch. Math. (Basel) \textbf{9} (1958), 102--109.

\bibitem[KOBS]{NicholsGrass}
A.~O. Krutov, R.~\'O~Buachalla, and K.~R. Strung, \emph{Nichols algebras and
  quantum principal bundles}, Int. Math. Res. Not. IMRN (2023), 42,
  \href{http://arxiv.org/abs/1701.04394}{{\ttfamily arXiv:1701.04394
  [math.QA]}}.

\bibitem[KS]{KSLeabh}
A.~Klimyk and K.~Schm\"udgen, \emph{Quantum groups and their representations},
  Texts and Monographs in Physics, Springer-Verlag, 1997.

\bibitem[Maj1]{MajidBocClass}
S.~Majid, \emph{Classification of bicovariant differential calculi}, J. Geom.
  Phys. \textbf{25} (1998), no.~1-2, 119--140,
  \href{http://arxiv.org/abs/q-alg/9608016}{{\ttfamily arXiv:q-alg/9608016
  [q-alg]}}.

\bibitem[Maj2]{Maj}
S.~Majid, \emph{Noncommutative {R}iemannian and spin geometry of the standard
  {$q$}-sphere}, Comm. Math. Phys. \textbf{256} (2005), 255--285,
  \href{http://arxiv.org/abs/math/0307351}{{\ttfamily arXiv:math/0307351
  [math.QA]}}.


\bibitem[Mey]{MEYER}
U.~Meyer, \emph{Projective quantum spaces}, Lett. Math. Phys. \textbf{35}
  (1995), no.~2, 91--97, \href{http://arxiv.org/abs/hep-th/9410039}{{\ttfamily
  arXiv:hep-th/9410039 [hep-th]}}.

\bibitem[OB1]{MMF1}
R.~\'{O}~Buachalla, \emph{Quantum bundle description of quantum projective
  spaces}, Comm. Math. Phys. \textbf{316} (2012), no.~2, 345--373,
  \href{http://arxiv.org/abs/1105.1768}{{\ttfamily arXiv:1105.1768 [math.QA]}}.

\bibitem[OB2]{MMF2}
R.~\'{O}~Buachalla, \emph{Noncommutative complex structures on quantum
  homogeneous spaces}, J. Geom. Phys. \textbf{99} (2016), 154--173,
  \href{http://arxiv.org/abs/1108.2374}{{\ttfamily arXiv:1108.2374 [math.QA]}}.

\bibitem[OB3]{MMF3}
R.~\'{O}~Buachalla, \emph{Noncommutative {K}\"{a}hler structures on quantum
  homogeneous spaces}, Adv. Math. \textbf{322} (2017), 892--939,
  \href{http://arxiv.org/abs/1602.08484}{{\ttfamily arXiv:1602.08484
  [math.QA]}}.

\bibitem[OBSvR]{OSV}
R.~\'O~Buachalla, J.~\v{S}\v{t}ovi\v{c}ek, and A.-C. van Roosmalen, \emph{A
  {K}odaira vanishing theorem for noncommutative {K}\"ahler structures},
  \href{http://arxiv.org/abs/1801.08125}{{\ttfamily arXiv:1801.08125
  [math.QA]}}, arXiv preprint, 2018.

\bibitem[PS]{PolishSch}
A.~Polishchuk and A.~Schwarz, \emph{Categories of holomorphic vector bundles on
  noncommutative two-tori}, Comm. Math. Phys. \textbf{236} (2003), no.~1,
  135--159, \href{http://arxiv.org/abs/math/0211262}{{\ttfamily
  arXiv:math/0211262 [math.QA]}}.

\bibitem[Sar]{Saraccoa19}
P.~Saracco, \emph{Antipodes, preantipodes and {F}robenius functors}, J. Algebra
  Appl. \textbf{20} (2021), no.~7, 2150124, 32,
  \href{http://arxiv.org/abs/1906.03435}{{\ttfamily arXiv:1906.03435
  [math.RA]}}.

\bibitem[Sch]{Schauenburg}
P.~Schauenburg, \emph{Hopf modules and {Y}etter-{D}rinfel'd modules}, J.
  Algebra \textbf{169} (1994), no.~3, 874--890.

\bibitem[Ser]{Serre1955}
J.-P. Serre, \emph{Faisceaux alg\'{e}briques coh\'{e}rents}, Ann. of Math. (2)
  \textbf{61} (1955), 197--278.

\bibitem[Skr]{Skryabin2007}
S.~Skryabin, \emph{Projectivity and freeness over comodule algebras}, Trans.
  Amer. Math. Soc. \textbf{359} (2007), no.~6, 2597--2623,
  \href{http://arxiv.org/abs/math/0610657}{{\ttfamily arXiv:math/0610657
  [math.RA]}}.

\bibitem[Swa]{Swan1962}
R.~G. Swan, \emph{Vector bundles and projective modules}, Trans. Amer. Math.
  Soc. \textbf{105} (1962), 264--277.

\bibitem[Tak]{Tak}
M.~Takeuchi, \emph{Relative {H}opf modules---equivalences and freeness
  criteria}, J. Algebra \textbf{60} (1979), no.~2, 452--471.

\bibitem[Wor]{WoroCQPGs}
S.~L. Woronowicz, \emph{Compact matrix pseudogroups}, Comm. Math. Phys.
  \textbf{111} (1987), no.~4, 613--665.

\end{thebibliography}

\end{document}